\documentclass[11pt]{amsart}

\usepackage{amsmath}
\usepackage{amssymb}
\usepackage{amsthm}
\usepackage{amstext}
\usepackage{color}

\usepackage{graphicx}

\DeclareMathOperator*{\Mod}{mod}

\newcommand{\hide}[1]{}
\newcommand{\eps}{\varepsilon}
\renewcommand{\phi}{\varphi}

\newcommand{\Pd}{\mathcal{P}_d}
\newcommand{\Sd}{\mathcal{S}_d}
\newcommand{\Sdp}{\mathcal{S}_{d,\rho}}

\newtheorem{lemma}{Lemma}
\newtheorem{theorem}[lemma]{Theorem}
\newtheorem{corollary}[lemma]{Corollary}

\newcommand{\Z}{\mathbb {Z}}
\newcommand{\R}{\mathbb {R}}
\newcommand{\C}{\mathbb {C}}
\newcommand{\Cbar}{\overline{\C}}
\newcommand{\disk}{\mathbb{D}}
\newcommand{\diskbar}{\overline{\mathbb{D}}}
\newcommand{\sm}{\setminus}
\newcommand{\ovl}[1]{\overline{#1}}
\renewcommand{\Re}{\text{\rm Re}}
\renewcommand{\Im}{\text{\rm Im}}

\newcommand{\lineclear}{\rule{0pt}{0pt}\nopagebreak\par\nopagebreak\noindent}

\theoremstyle{remark}
\newtheorem{Remark}{Remark}

\title[Universal Starting Points for Newton's Method]{A small
probabilistic universal set of starting points for finding roots of
complex polynomials by Newton's method}

\author{B\'ela Bollob\'as, Malte Lackmann, Dierk Schleicher}
\address{Department of Pure Mathematics and Mathematical Statistics,
University of Cambridge, Cambridge CB3 0WB, UK; and
Department of Mathematical Sciences,
University of Memphis, Memphis TN 38152, USA.}
\email{B.Bollobas@dpmms.cam.ac.uk}

\address{Immenkorv 13, D-24582 Bordesholm, Germany}
\email{malte.lackmann@web.de}
\address{Research I, Jacobs University Bremen, Postfach 750 561,
D-28725 Bremen, Germany}
\email{dierk@jacobs-university.de}

\begin{document}

\begin{abstract}
We specify a small set, consisting of $O(d(\log\log d)^2)$ points,
that intersects the basins under Newton's method of \emph{all} roots
of \emph{all} (suitably normalized) complex polynomials of fixed
degrees $d$, with arbitrarily high probability. This set is an
efficient and universal \emph{probabilistic} set of starting points
to find all roots of polynomials of degree $d$ using Newton's method;
the best known \emph{deterministic} set of starting points consists
of $\lceil 1.1d(\log d)^2\rceil$ points.
\end{abstract}

\maketitle

\section{Introduction}

Newton's root-finding method is as old as analysis, but still not
well understood, even in the fundamental case of finding all roots of
a polynomial in a single variable. Its local convergence properties
are well known; near simple roots convergence is quadratic and thus
extremely rapid. However, the global dynamical properties are
insufficiently understood so that numerical analysis algorithms often
use different global methods, and resort to Newton's method for a
final local ``polishing'' of the roots.

This article is a contribution towards a better understanding of the
global properties of Newton's method, applied to polynomials in a
single complex variable. Even for polynomials over the reals, and
even if all the roots are real, it is often preferable to use complex
methods; see Figure~\ref{Fig:NewtonDynamics}.

\begin{figure}
\includegraphics[width=0.49\textwidth]{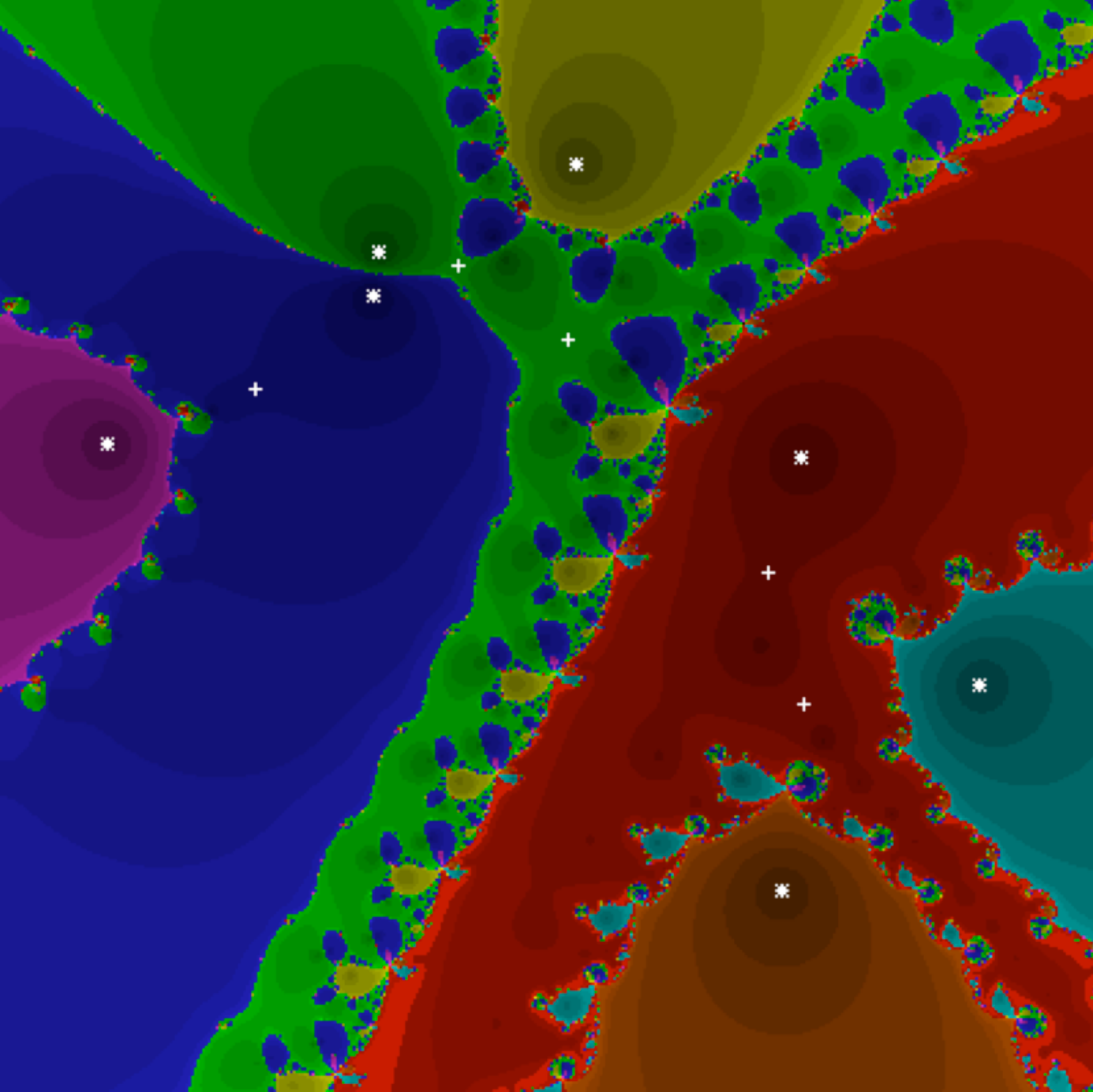}
\includegraphics[width=0.49\textwidth]{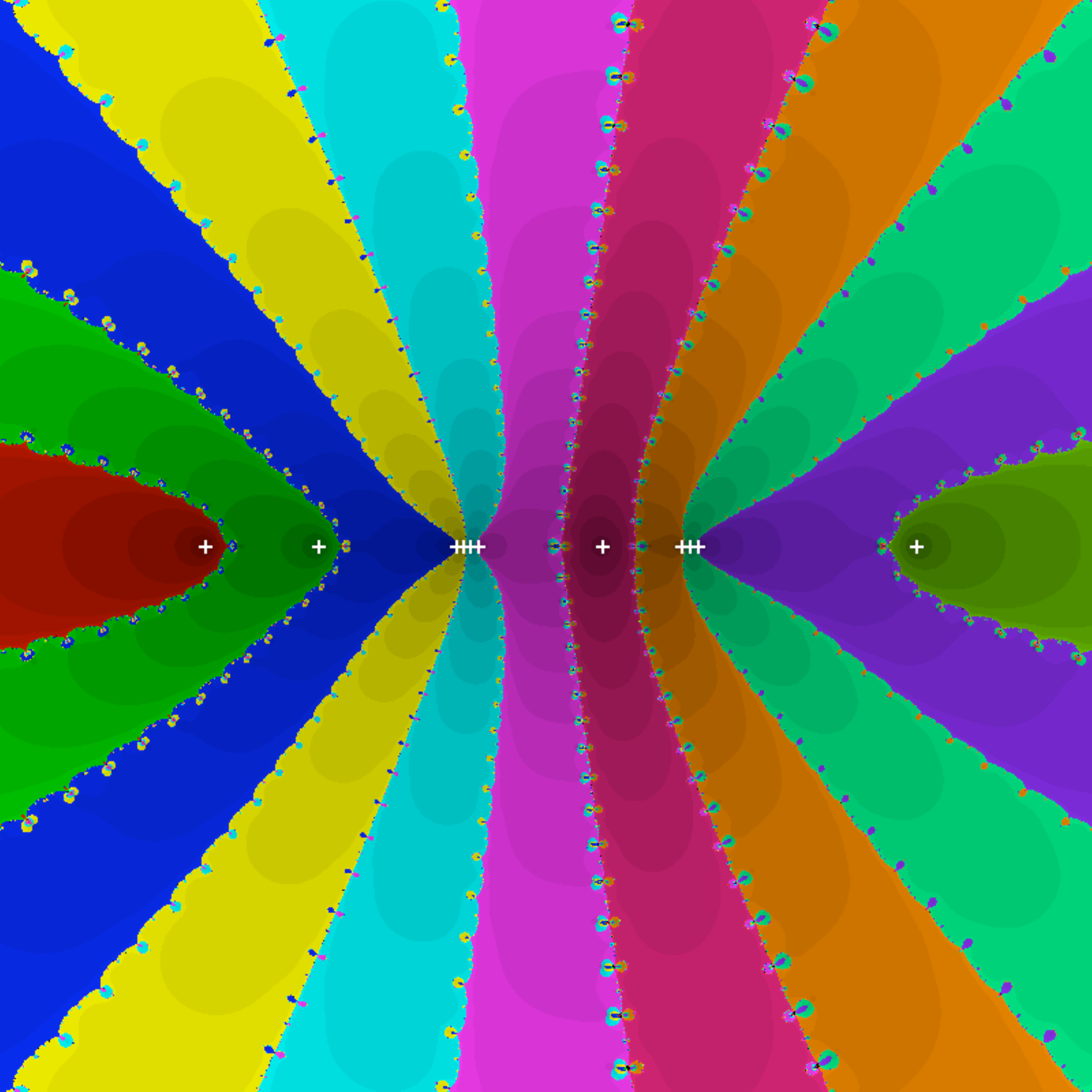}
\caption{Dynamical planes of Newton maps of two complex polynomials.
Different colors illustrate basins of attraction of different roots;
shades of color illustrate different speeds of convergence. It is
clearly visible that all immediate basins are unbounded and have one
or several channels to $\infty$ of different widths. Left: a
polynomial of degree $7$. Right: a polynomial of degree $11$ with all
roots real. Some of the roots are very close to each other; however, away
from the disk containing all the roots, the basins and their channels
all have almost uniform width, so that finding the real roots using
complex methods is much easier.}
\label{Fig:NewtonDynamics}
\end{figure}

\begin{figure}
\includegraphics[width=0.51\textwidth,trim=0 130 0 60,clip]{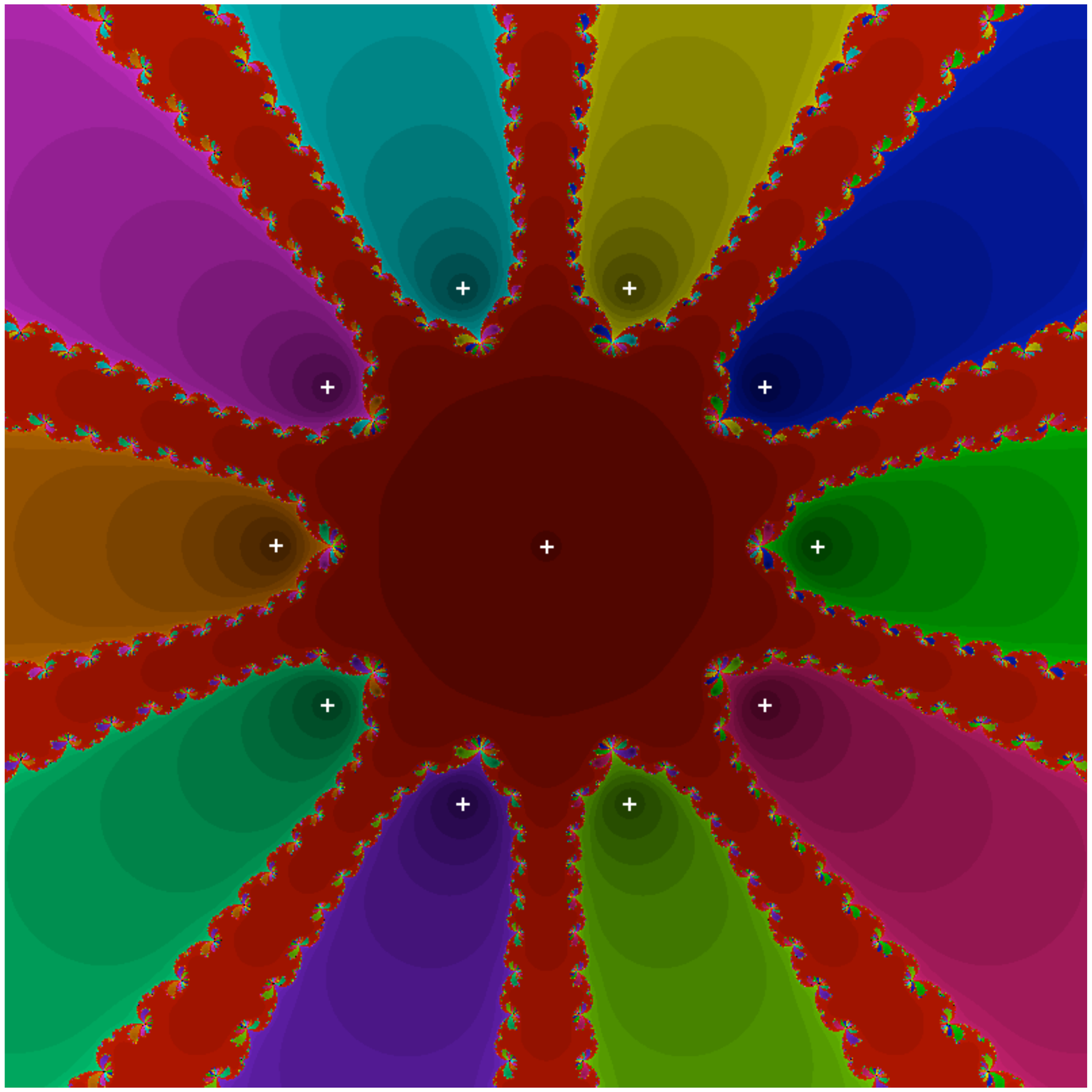}
\caption{The dynamical plane of the polynomial $p(z)=z(z^{10}-1)$: the ten roots of unity each have one ``thick'' channel, while the root $z=0$ has 10 channels (red) which are all rather ``thin''. The deterministic method from \cite{HSS} would search for the individual thin channels and thus requires more points, while our method searches for the union of all thin channels, which together are much bigger.}
\label{Fig:NewtonDynamicsThinChannels}
\end{figure}

Among the difficulties with Newton's method are the following:
\begin{itemize}
\item
if an orbit under iteration comes close to a critical point of the
polynomial, the Newton map sends the orbit far away near $\infty$, so
that control of the dynamics is lost, and in any case a large number
of iterations are required until the orbit comes back to where the
roots are;
\item
there are polynomials with open sets of starting points that do not
converge to any root (Smale~\cite{Smale} asked, in 1984, for a
classification of such polynomials; an answer has recently been given by
Mikulich in current work \cite{ZhenyaThesis});
\item
the boundary of the basins of convergence for the roots may have
positive planar Lebesgue measure (this follows from recent work by
Buff and Ch\'eritat on the existence of Julia sets with positive
measure \cite{BuffCheritat}, combined with Douady and Hubbard's
renormalization theory \cite{PolyLike});
\item
even if almost every point in $\C$ converges to some root under the
Newton iteration, our goal is to find \emph{all} roots of the
polynomial, and with bounded complexity. Finding some roots and
deflating is usually not an option, because deflation is in general
numerically unstable (unless the roots are found in a specific
order), and because deflation might not be compatible with the way
the polynomial may be specified, or evaluated efficiently (for
instance, if the polynomial itself is given by an efficient iteration
procedure).
\end{itemize}
See \cite{Johannes} for a recent survey of known results on Newton's method.

This article is a contribution towards the goal of turning Newton's method into an efficient algorithm.
To achieve this goal, one should:
\begin{itemize}
\item
select a finite set $\Sd$ of \emph{good starting points} that are
guaranteed to intersect the basins of all roots;
\item
specify a \emph{condition when to stop iterating} any of these starting
points, because the orbit is either sufficiently close to a root, or
the orbit is discarded in favor of some other starting points;
\item
give a good \emph{bound on the complexity} of Newton's method to find all
roots of the polynomial with prescribed precision.
\end{itemize}
This article is concerned with the first of these questions; we will not discuss the other two issues in detail (see for instance \cite{Smale2,Johannes}). Concerning efficiency of the Newton method, we mention the following recent result from \cite{NewtonEfficient,NewtonGlobal,NewtonTodor}:
roughly speaking, for ``most'' polynomials $p$ of degree
$d$, properly normalized, our universal set $\Sd$ contains $d$ points
that converge to the $d$ different roots of $p$ so that the total
number of Newton iterations, for all $d$ roots combined, to achieve
an accuracy of $\eps$ is at most $O(d^2\log^4d)+d\log|\log\eps|$.
This makes it possible to turn Newton's method into an \emph{efficient} algorithm for the problem of finding all roots of a given polynomial. 

To state our main result, let $\Pd$ be the space of polynomials of
degree $d$, normalized so
that all roots are contained in the complex unit disk $\disk$.

\begin{theorem}[Small Probabilistic Universal Set of Starting Points]
\label{Thm:Main} \lineclear
For every degree $d\ge 3$, there is an explicit universal
probabilistic set $\Sd$ consisting of $O(d(\log\log d)^2)$ starting
points so that for every polynomial $p\in\Pd$, the probability is
greater than $1/2$ that the immediate basin
of each root of $p$ contains at least one point in $\Sd$
(in fact, this probability is greater than $1-1/d\ge 2/3$).
\end{theorem}


\begin{Remark}
The meaning of an ``explicit and
universal'' probabilistic set is as follows: we give an explicit
probability distribution of starting points that depends only on $d$
so that for any $p\in\Pd$, with probability at least $1-1/d$ all
immediate basins contain at least one point in this set.
(The probability $1-1/d$ may seem somewhat artificial; it is what we get naturally from of our estimates, and it is better than the uniform $2/3$.) 
Of course,
enlarging this set of points appropriately, the probability of success can be increased (see Remark \ref{higherProb}):
\emph{For every probability $\rho\in(0,1)$, there is an explicit and universal set $\Sdp$ of starting points with cardinality 
$O\big(d (\log\log d)^2  + d|\log(1-\rho)| \big)$
such that the statement of Theorem~\ref{Thm:Main} is true with probability $\rho$ instead of $1-1/d$.}
\hide{Further enlarging this set by a constant
factor, one can make sure that there are starting points ``well
within'' the basins, so that the number of iterations required to
find all roots with prescribed precision is small
\cite{NewtonEfficient,NewtonGlobal}.}

This result is in a similar spirit as \cite{HSS}, where a similar
explicit universal set of starting points is constructed.
It consists of $\lceil 1.1d (\log d)^2\rceil$ points and is deterministic. Our new
set is significantly smaller than the deterministic set, much closer
to the ``ideal lower bound'' of $d$ points, but we can do so only
using a probabilistic set. We believe that there is no deterministic
explicit and universal set of starting points with $o(d\log d)$
points.
\medskip

\noindent
\emph{Construction of the set $\Sd$.}
Our set $\Sd$ is constructed as follows: firstly, we define a ``fundamental annulus'' $V:=\left\{z\in\C\colon R\sqrt{1-1/d}<z<R\right\}$ for some
$R>1+\sqrt2$, and choose a
``deterministic set'' of approximately $(16/\pi) d (\log\log d)^2$
points that are distributed on $m=\lceil (2/\pi)\log\log d\rceil$
circles. These circles have radii $R_k=R(1-1/d)^{(k-1/2)/2m}$ for
$k=0,1,,\dots,m-1$, and each circle contains $\lceil4\pi
d\lceil(2/\pi) \log\log d\rceil\rceil$ points at equal distances.
This construction is in principle the same as in \cite{HSS}.
Secondly, we choose a ``probabilistic set'' of $\lceil(300/\pi) d 
\log\log d\rceil$ points randomly inside the annulus 
$A_R=\{z\in\C\colon R(d-1)/d-1/d<|z|<R\}$ for some $R\ge 11$.
These deterministic and probabilistic sets of points will
respectively find  ``thick'' and ``thin'' roots, as defined below. Iterating Newton's method starting at these points (in parallel or in any order), we will find all roots of $p$ with   probability at least $1-1/d$
(or with any probability $\rho\in(0,1)$ when taking appropriately more points in the probabilistic set).
\end{Remark}

\noindent
\emph{Historical Remark.}
This research has its origins at the 50th anniversary celebration of
the International Mathematical Olympiad (IMO) held in 2009 in Bremen,
Germany.
One chief goal of this celebration was to bring together olympiad
mathematics and research mathematics, and people involved in both.
This paper was authored by a research mathematician who in his
youth was one of the first contestants ever at IMOs and in 2009 was a
guest of honor at the 50th IMO, together with one of the contestants
there, and a research mathematician who was among the
senior organizers of that IMO and its anniversary. This work is thus
very much in the spirit of the IMO anniversary, and we are grateful
to this anniversary celebration that has brought us together.

We gratefully acknowledge partial support through NSF grants DMS-0906634,
CNS-0721983 and CCF-0728928, ARO grant W911NF-06-1-0076, and the  TAMOP-4.2.2/08/1/2008-0008 program of the Hungarian National Development Agency (BB), as well as the European Research and Training network CODY, the ESF programme HCAA, and the German Research Council DFG (DS). We also gladly acknowledge useful feedback from an anonymous referee.

\section{Channels and Their Moduli}

Consider a complex polynomial $p(z)=c\prod_{j=1}^d(z-\alpha_j)$ and
let $N_p(z)=z-p(z)/p'(z)$ be the associated Newton map. This is a
rational map of degree $d$ if all roots of $p$ are distinct, and of
lower degree otherwise. Without changing the Newton map, we may suppose
that $c=1$, and after rescaling, we may suppose that all
$\alpha_j\in\disk$.

For any root $\alpha$ of $p$, let $U_\alpha$ be the \emph{immediate basin}
of $\alpha$: the basin is the set of all $z\in\C$ that converge to
$\alpha$ under iteration of $N_p$, and the immediate basin is the
connected component containing $\alpha$. It is known that each
$U_\alpha$ is simply connected \cite{Prz} and that the restriction of
$N_p$ to $U_\alpha$ sends $U_\alpha$ to itself as a proper map of
some degree $k+1\in\{2,3,\dots,d\}$. We will use the construction and
some results from \cite{HSS}.
If $\phi\colon U_\alpha\to\disk$ is a Riemann map with
$\phi(\alpha)=0$, then $f:=\phi\circ N_p\circ\phi^{-1}$ is a proper
holomorphic self-map of $\disk$ of degree $k+1$ and thus extends, by
Schwarz reflection, to a rational map of degree $k+1$, and the
restriction of $f$ to $\partial\disk$ is a covering of
$\partial\disk$, also of degree $k+1$. In particular, the restriction
of $f$ to $\partial\disk$ has $k\ge 1$ fixed points
$q_1,\dots,q_{k}$. Set $\lambda_i:=f'(q_i)$, for $i=1,2,\dots,k$.

The holomorphic fixed point formula (which essentially is the residue
theorem for $1/(z-f(z))$; see \cite{MiBook}) implies that
\begin{equation}
\sum_{i=1}^{k} \frac{1}{\lambda_i-1} \ge 1
\label{Eq:FixedPointIndices}
\end{equation}
(with equality if the root $\alpha$ is simple). Each of these $k$
fixed points gives rise to a \emph{channel} to $\infty$ in the
immediate basin $U_\alpha$: for our purposes, a channel is an unbounded
component $B_i$ of $U_\alpha\sm\diskbar$. Near $\infty$, each channel
is mapped by $N_p$ conformally to itself, and it defines an access to
$\infty$ within $U_\alpha$ that is fixed by $N_p$. The quotient of
$B_i$ by the dynamics of $N_p$ is a conformal annulus with modulus
$\mu_i=\pi/\log\lambda_i$.

Choose some positive real number $M < \pi / \log 4\approx 2.266$ that
will be specified later (we will eventually use $M=\pi/\log\log d$
for large $d$).

We call a root $\alpha$ \emph{thick} if it has a channel with
modulus $\mu_i\ge M$, and \emph{thin} if there is no such channel. We
will treat these two cases separately.

\begin{itemize}
\item
We will explicitly and deterministically construct a set of
$\lceil 4\pi d\lceil 2/M\rceil^2 \rceil$ points that is guaranteed to
intersect each channel of a root
with modulus greater than $M$. This set will thus suffice to ``find''
all thick roots.
\item
The advantage of thin roots is that even though the individual
channels have small moduli, the total area of these channels within
any fundamental domain of the Newton dynamics is greater than in the
thick case: each channel may have little area, but there are more channels
in this case (see Figure~\ref{Fig:NewtonDynamicsThinChannels}). We 
show that if $\lceil 300\, d\log d/ Me^{\pi/M}\rceil$
points are distributed randomly in a certain fundamental annulus of 
the Newton dynamics, then the probability that the immediate basins 
of \emph{all} thin roots contain such a point is at least $1-1/d$.
\end{itemize}

\begin{Remark}
If $\alpha$ is a thin root, then all  $\mu_i<M$, hence all
$\lambda_i-1=e^{\pi/\mu_i}-1>e^{\pi/M}-1$, so by
(\ref{Eq:FixedPointIndices}), the number $k$ of channels of a thin
root is strictly greater than $e^{\pi/M}-1$. But the mapping degree
of $U_\alpha$
equals $k+1$, so $U_\alpha$ must contain $k > e^{\pi/M}-1$ of the at
most $2d-2$
critical points of $N_p$, and thus the number of thin roots is at
most $(2d-2)/(e^{\pi/M}-1)$.
In the end, we will use $M=\pi/\log\log d$, so the number of thin
roots will be at most $(2d-2)/(\log d-1)$: most roots will be thick.
It seems to be an interesting question (outside the scope of this paper) to
estimate how likely it is for a given polynomial of degree $d$ to have \emph{all} its roots thick.

If there are thin roots, then we can estimate
\begin{equation}
e^{\pi/M} < k+1 \le d \;;
\label{Eq:ThinRootCondition}
\end{equation}
in particular, there are no thin roots
at all if $M\le \pi/\log d$.
\end{Remark}

A \emph{conformal quadrilateral} is a Riemann domain $Q\subset\C$
with two distinguished connected and disjoint subsets of the boundary. In our
setting, the boundary of $Q$ may not be a topological curve, but the
two distinguished boundary subsets will be; we will call them
\emph{distinguished boundary arcs}. Then there is a unique
$h>0$ so that the domain $Q_h:=\{z\in\C\colon 0<\Im\,z<1,
0<\Re\,z<h\}$ has a Riemann map $\phi\colon Q\to Q_h$ that maps the
two distinguished boundary arcs onto the two horizontal sides of
$Q_h$ (the Riemann map may not extend continuously to the boundary of
$Q$, but it does so near the two distinguished boundary arcs; the
general framework of extremal length using curve families works even
if the boundaries are not curves). The value $h$ is defined as the
\emph{conformal modulus} of the quadrilateral $Q$ with respect to the
two boundary subsets, and denoted $\Mod(Q)$; it is invariant for
conformal homeomorphisms that respect the distinguished boundary subsets, in particular for Riemann maps with this property 
\cite{Ahlfors}.

Identifying the two distinguished boundary arcs, we obtain a complex annulus (a doubly connected Riemann surface) with modulus $\Mod(Q)$ or less (the exact modulus depends on how the boundaries are identified).

\section{Hitting thick roots}

\label{Sec:thick}

In this section, we will construct an explicit and deterministic set
of starting points that is guaranteed to intersect the basins of all
thick roots.
Our arguments are essentially the same as in \cite[Section~5]{HSS},
except that we no longer need to find all roots, but only the thick
ones.

If $R>(d+1)/(d-1)$ and $C_R$ is the circle of radius $R$ centered at
the origin, then $N_p$ maps $C_R$ homeomorphically onto some
topological circle around $\disk$, and there is some $\kappa>0$ so
that the round annulus
\[
V_{R,\kappa,d}=\left\{z\in\C\colon
R \left(\frac{d-1}{d}\right)^\kappa <|z|< R\right\}
\]
is contained in the topological
annulus between $C_R$ and $N_p(C_R)$; specifically, if $R\ge 1+\sqrt
2$, then $\kappa\ge 1/2$ for all $d$.
If $R$ tends to $\infty$, then $\kappa$ tends to $1$.
All this is \cite[Lemmas~4 and 12]{HSS}; see also Figure~\ref{Fig:Channels}.

\begin{figure}[htbp]
\includegraphics[width=1.0\textwidth,trim=0 0 0 0,clip]{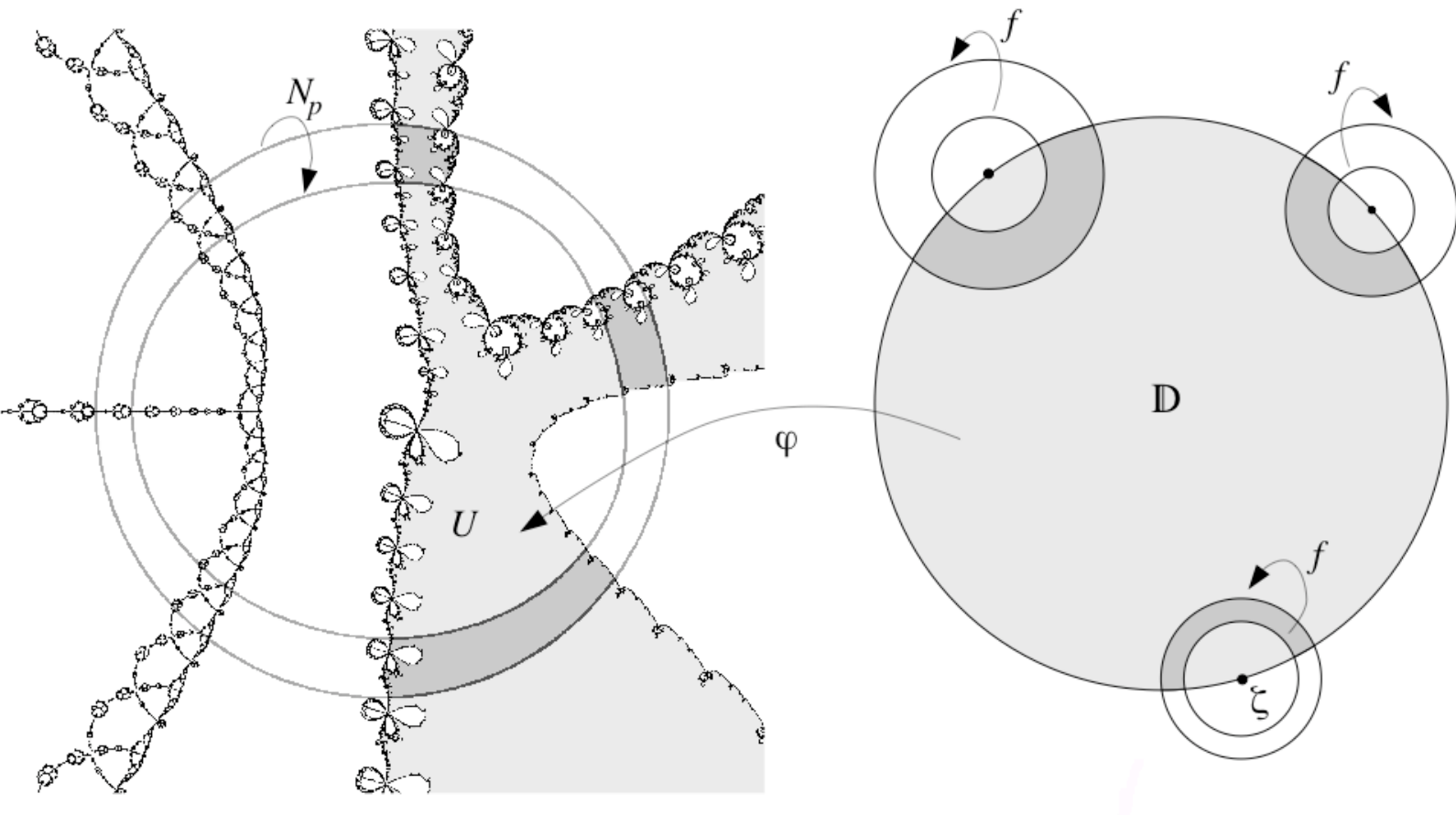}
\caption{Left: the dynamics of Newton's method for some complex
polynomial. Highlighted is the immediate basin of attraction of one
root, with fundamental domains within the channels shaded. Also shown
is the circle at radius $R$ and its image, which is a topological
(but not geometric) circle. Right: the complex unit disk $\disk$
provides a conformal model for the Newton dynamics of the immediate
basin. (Picture taken from \cite{HSS}.)}
\label{Fig:Channels}
\end{figure}

We will use the round annulus $V=V_{R,\kappa,d}$ with $R\ge 1+\sqrt
2$ and $\kappa=1/2$ (if we use larger values of $R$, then we can take
larger values of $\kappa$, and our bounds will eventually be slightly
better; however, in practice these starting points would be further
away from the roots, and the iteration would take longer).

\begin{Remark} \label{Rem:AreaModulusV}
The modulus of $V$ is $|\log((d-1)/d)|/4\pi>1/4\pi d$.
\end{Remark}

Consider some channel $B_i$. We want to define $Q_i$ as ``the part of
the channel $B_i$ within $V$''. If each of the two boundary circles
of $V$ intersects $B_i$ in a single connected arc, we set
$Q_i:=B_i\cap V$. However, if $B_i\sm V$ has more than two connected
components (see Figure~\ref{Fig:Qi}), we need to be more careful.
Consider the intersection of $B_i$ with $C_R$, the outer boundary of
$V$. Let $\gamma$ be any connected component in this intersection. It
separates $U_\alpha$ into two components, one of which contains the
root $\alpha$; then $\gamma$ will be called an \emph{essential} boundary
arc of $B_i\cap C_R$ if the component of $U_\alpha\sm\gamma$ not
containing $\alpha$ is unbounded: this means that $\gamma$ separates
the unbounded part of the channel $B_i$ from the root. At least one
component of $B_i\cap C_R$ is essential; choose one such essential
component $\gamma$, let $\gamma':=N_p(\gamma)$, and let $Q'_i$ be the
subset of $U_\alpha$ that is bounded by $\gamma$ and $\gamma'$ (if
$B_i$ intersects $C_R$ and equivalently $N_p(C_R)$ in only one
component, then $Q'_i$ is the part of $B_i$ between $C_R$ and
$N_p(C_R)$; in general, the difference may consist of some number of
bounded components). Then $Q'_i$ is a fundamental domain of $B_i$ by
the dynamics; when viewed as a quadrilateral with distinguished
boundary arcs $\gamma$ and $\gamma'$, then
$\Mod(Q'_i)\ge\Mod(B_i)=\mu_i$ ($Q'_i$ is a quadrilateral, the
modulus of $B_i$ is defined using the quotient annulus of $B_i$ by
the dynamics).

Now let $C_{R'}$ be the inner boundary circle of $V$ and consider all
essential arcs of intersection of $B_i\cap C_{R'}$. If there is only
one, then let $\gamma''$ be this essential arc. If there are several,
then they are totally ordered (because they all separate $\alpha$ in
$U_\alpha$ from the unbounded component of $B_i\sm V$). Let
$\gamma''$ be the outermost component that separates $\alpha$ from
$\gamma$ (i.e., the one closest to $\gamma$), and let $Q_i$ be the component of
$B_i\sm(\gamma\cup\gamma'')$ that is bounded by $\gamma$ and
$\gamma''$. This is a conformal quadrilateral with $Q_i\subset Q'_i$,
and with $\gamma$ and $\gamma''$ as distinguished boundary arcs, and we
have $\Mod(Q_i)\ge\Mod(Q'_i)\ge\mu_i$.

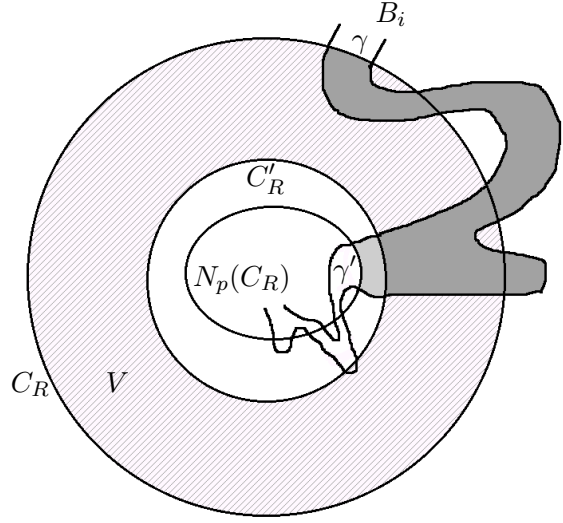
\begin{figure}[htbp]
\includegraphics[width=100mm,trim=40 60 40 60,clip]{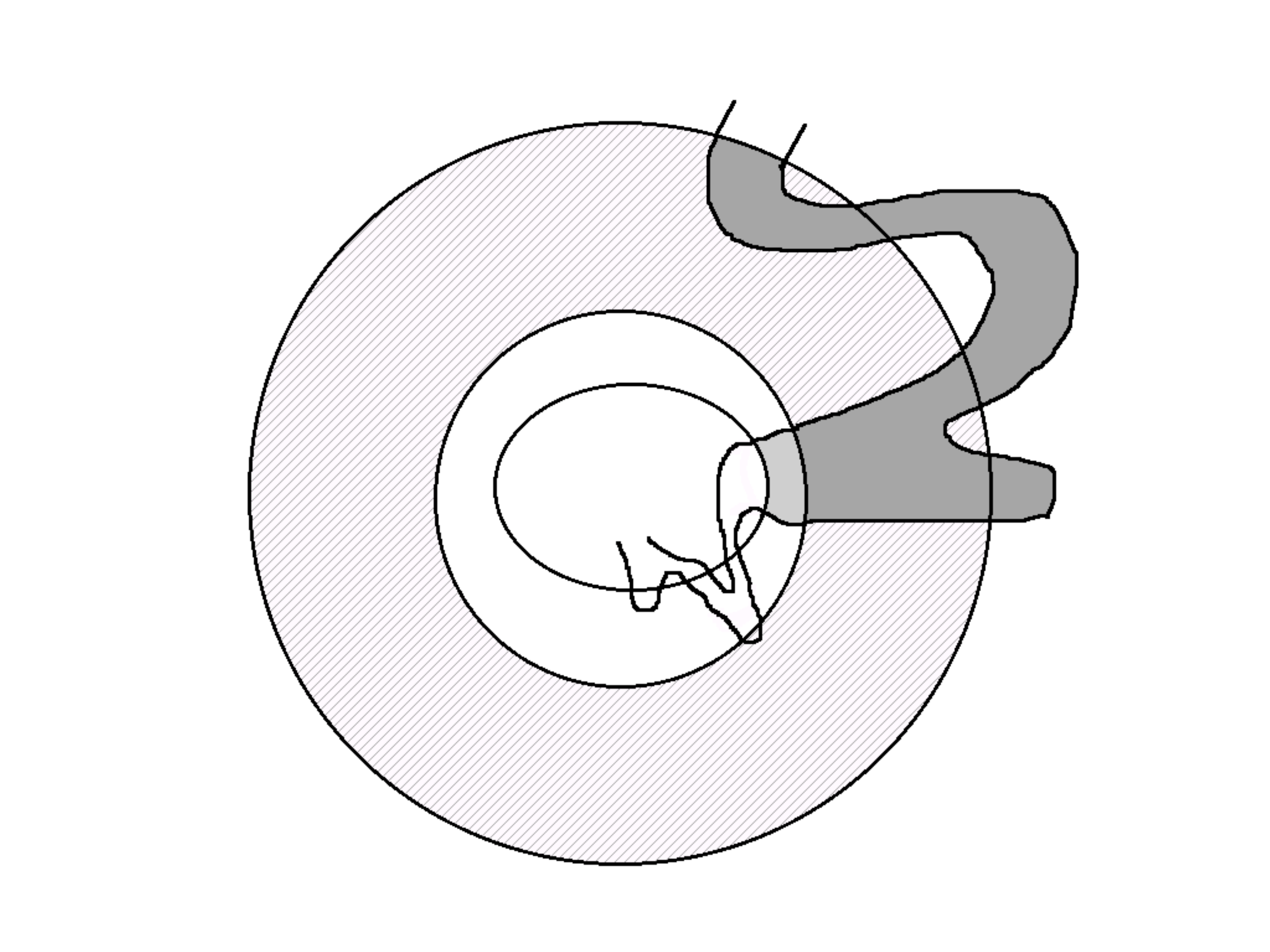}
\begin{picture}(0,0)
\put(-210,50){$V$}
\put(-246,50){$C_R$}
\put(-157,128){$C'_R$}
\put(-177,90){$N_p(C_R)$}
\put(-108,190){$B_i$}
\put(-117,180){$\gamma$}
\put(-124,93){$\gamma'$}
\end{picture}
\label{Fig:Qi}
\caption{The annulus $V$ (hatched). Its outer boundary circle is
$C_R$; the image $N_p(C_r)$ is a topological circle within the
bounded complementary component of $V$. Also shown is a channel
$B_i$; it intersects $C_R$ in four arcs, three of which are
essential. Shaded is the quadrilateral $Q'_i$ which is bounded by two
essential arcs, one on $C_R$ and one on $N_p(C_R)$; it is a
fundamental domain of $B_i$ modulo $N_p$. The quadrilateral
$Q_i\subset Q'_i$ is shaded darker: it is bounded by two essential
arcs on $\partial V$, but may not be contained in $V$. }
\end{figure}

Our task will be to distribute sufficiently many points into $V$ so
that we hit quadrilaterals $Q_i\subset V$ with moduli bounded below.

\begin{lemma}
\label{Lem:UniversalGeometry}
Let $S=\{z\in\C\colon -1/2<\Re\,z<1/2\}$ and let $Q\subset \C$ be a
quadrilateral whose two distinguished boundary arcs are on the two
vertical sides of $S$, one on each. Suppose that $Q$ is disjoint from the set
$i\Z$. Then the modulus of $Q$ is at most $2$.
\end{lemma}
\begin{proof}
This is an easy extremal length exercise \cite{Ahlfors}. There is
an integer $n\in\Z$ so that any curve in $Q$ connecting the two
distinguished boundary arcs must intersect the segment $[ni,(n+1)i]$.
Without loss of generality, suppose that $n=0$.

Let $B:=\{z\in S\colon -1/2<\Im\,z<3/2\}$ and let $\rho$ be the
characteristic function of $B$. Then for any curve $\gamma\subset Q$
connecting the two distinguished boundary arcs, its intersection with
$B$ has length at least $1$. Since $\int_\C \rho^2\,dx\,dy=2$, it
follows that $\Mod(Q)\le 2$.
\end{proof}
\begin{Remark}
The bound of $2$ is not sharp. It is not hard to calculate the exact
bound \cite{Ahlfors}, but we are not optimizing constant factors here.
\end{Remark}

\begin{lemma}
\label{Lem:ConstructionThick}
If $V$ is subdivided into at least $2/M$ concentric and conformally
equivalent subannuli, and at least $4\pi d\lceil 2/M \rceil$ points
are distributed onto the core circles of all subannuli, so that the
points on all circles are equidistributed, then each quadrilateral
$Q_i$ with modulus at least $M$ contains at least one of these points.
\end{lemma}
\begin{proof}
Let
$m:=\lceil 2/M\rceil$ and subdivide $V$ into $m$ concentric and
conformally equivalent subannuli $V_1,\dots,V_{m}$, ordered by
decreasing radii (so that $V_k=\{z\in V\colon R\beta^{k} < |z| <
R\beta^{k-1}\}$ for $\beta=(1-1/d)^{1/2m}$)
. Write $Q$ for $Q_i$; this is a quadrilateral for
which the two distinguished boundary arcs are on
$\partial V$, one on each boundary component of $V$.

Subdivide $Q$ into quadrilaterals $Q'_1,\dots,Q'_{m}$ as follows,
similarly as above. The
common boundary circle of $V_j$ and $V_{j+1}$ may intersect $Q$ in
several arcs; such an arc is essential if it separates the root
$\alpha$ from the unbounded component of $B_i\sm V$. Use an essential
arc to separate $Q'_j$
from $Q'_{j+1}$, for $j=1,2,\dots,m-1$. (In the special case that
$B_i\cap \partial V_j$ only has two connected components, then simply
$Q'_j=B_i\cap V_j$.)

By the Gr\"otzsch inequality, one of the quadrilaterals $Q'_j$ has
modulus $\Mod(Q'_j)\ge m\cdot\Mod(Q)\ge \lceil 2/M\rceil M\ge 2$.
Supposing for now that $0\not\in Q'_j$ and taking
logarithms, the annulus $V_j$ becomes an infinite vertical strip of
width $|\log((d-1)/d)|/2m>1/2md$, and $Q'_j$ becomes a
quadrilateral that connects the two boundary sides of the strip; see
Figure~\ref{Fig:PointGrid}.

\begin{figure}
\includegraphics[width=0.8\textwidth]{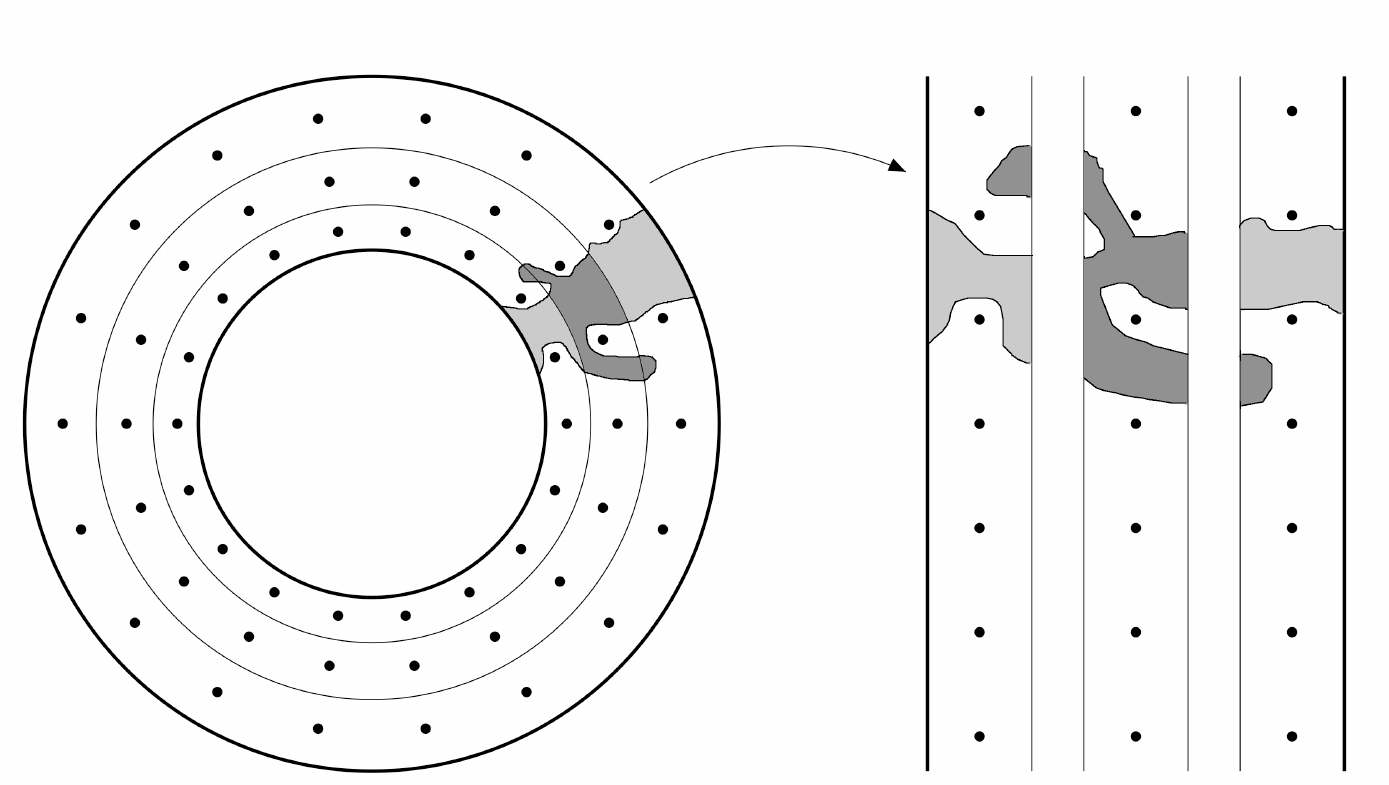}
\begin{picture}(100,000)
\put(60,150){$\log$}
\end{picture}
\caption{The annulus $V$ is subdivided into $m=3$ concentric
subannuli, all of equal moduli. The logarithm unfolds these annuli to
vertical strips (moved apart to show them separately). Highlighted is
the intersection of one channel with $V$. The quadrilateral in the
channel corresponding to the middle subannulus is shown in a darker
shade: notice that it intersects the other subannuli as well.
}
\label{Fig:PointGrid}
\end{figure}
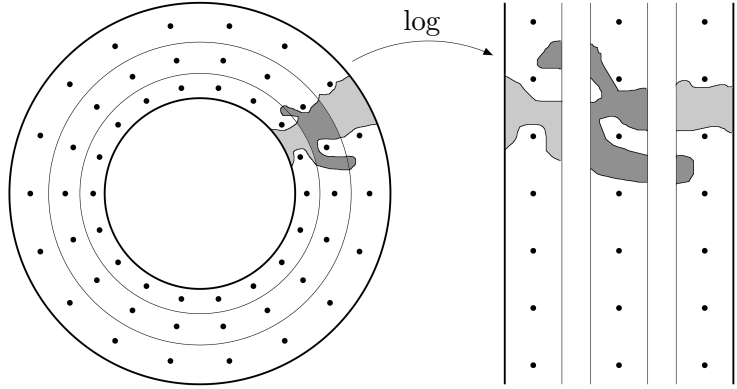

By Lemma~\ref{Lem:UniversalGeometry}, appropriately rescaled, each
annulus of modulus $2$ intersects the central vertical line within
this strip in a straight line segment of length at least $1/2md$.
Therefore, placing an infinite sequence of points on any vertical
line within the strip so that adjacent points have
distance less than $1/2md$, one can be sure that at least one of
these points intersects the annulus. The exponential map projects the
strip back onto $V_j$ as a universal cover and has period $2\pi i$,
so the required number of points on $V_j$ is $4\pi md=4\pi
d\lceil 2/M\rceil$.

If $Q'_j$ happens to contain the point $z=0$, then one cannot take the
log of $Q'_j$; but one can take the log of $Q'_j\cap V_j$ and transport
the function $\rho$ in the proof of Lemma~\ref{Lem:UniversalGeometry}
into $Q'_j\cap V_j$. This suffices for the conclusion to remain valid.
\end{proof}

\begin{corollary}[Deterministic Starting Points for Thick Roots] \lineclear
For every $d$ there is an an explicit set consisting of $\big\lceil 4\pi d\lceil
2/M\rceil \big\rceil\lceil 2/M\rceil\approx 16\pi d/M^2$ points in $V$ so
that for each $p\in\Pd$ and each thick
root of $p$, at least one point in $\Pd$ is contained in the
immediate basin of this root.
\end{corollary}
\begin{proof}
Using the construction described in Lemma
\ref{Lem:ConstructionThick}, we have $m = \lceil 2/M\rceil$ circles,
and each circle contains $\lceil 4\pi d\lceil 2/M\rceil\rceil$
points. Hence the total number of required points is as claimed.
These points intersect each quadrilateral $Q_i$ and thus the
immediate basin of each thick root.
\end{proof}

\section{Hitting Thin Roots}

Our goal in this case is to find a good lower bound for the area of
the union of all channels of any root, guaranteeing us that we will hit one of
the channels with high probability if we distribute sufficiently many
points randomly on a specified annulus.
The area of intersection of a channel with modulus $\mu_i$ with an
annulus will be bounded below by some multiple of $\mu_i$, so the total area of
intersection of an immediate basin with the annulus will be
proportional to $\sum \mu_i$, summed over all channels of the root.
We thus start with a lower bound for
$\sum_{i=1}^k \mu_i$.

Set $a_i = \frac{1}{\lambda_i-1}$, so that $\sum_{i=1}^k a_i \ge 1$. We have \[
\mu_i = \frac \pi {\log \lambda_i} = \frac{\pi}{\log(1+\frac1{a_i})}\;.
\]
Since $\mu_i < M$ for all $i$, we get that $a_i < 1/(e^{\pi/ M}-1)$
for all $i$.

We want to find a lower bound for
\[
\sum_{i=1}^k \mu_i
=\sum_{i=1}^k \frac{\pi}{\log(1+1/a_i)}
\]
subject to the conditions $\sum_{i=1}^k a_i \ge 1$ and $a_i <
1/(e^{\pi /M}-1)$.

\begin{lemma}
\label{Lem:Concave}
The function $f\colon \R^+\to\R^+$, $f(x)=\pi/\log(1+1/x)$ is strictly
monotonically increasing and concave (i.e., its graph is above the
line segment through any two points on it).
\end{lemma}
\begin{proof}
It suffices to prove that $f'$ is positive and monotonically decreasing.
This is a straightforward exercise.
\end{proof}

\begin{lemma}
\label{Lem:area}
If $\mu_i<M$ for all $i\in \{1, \ldots k\}$, then $\sum_{i=1}^k \mu_i
>     \frac 12 M e^{\pi/M}$.
\end{lemma}

\begin{proof}
Without loss of generality, assume that $a_1 \ge a_2 \ge \ldots \ge
a_k$, and that $\sum a_i=1$. We now consider the sequence $(b_1,
\ldots b_k)$ defined by
\begin{eqnarray*}
b_i = \left\{\begin{array}{cl}  \frac{1}{e^{\pi/ M} -1 } & \text{if
} i \le \lfloor e^{\pi/ M} -1 \rfloor
\\
1 - \frac{\lfloor e^{\pi/ M} -1 \rfloor}{e^{\pi/ M} -1} & \text{if  }
i = \lfloor e^{\pi/ M} -1 \rfloor+1
\\
0 & \text{if  } i > \lfloor e^{\pi/ M} -1 \rfloor+1\;. \end{array}   \right.
\end{eqnarray*}
Then we also have $\sum b_i=1$, and since all $a_i <
\frac{1}{e^{\pi/M} - 1}$, it follows that the sequence $(b_1, b_2,
\ldots b_k)$ majorizes the sequence $(a_1, a_2, \ldots a_k)$, in the
sense that
\[
\sum_{i=1}^m b_i \ge \sum_{i=1}^m a_i
\]
for all $m\in\{1,2,\dots,k\}$, with equality for $m=k$.
Since the function $f$ is concave by Lemma~\ref{Lem:Concave}, we get
from Karamata's inequality (see \cite[Thm. 108]{Ineq}) that $\sum
f(a_i)\ge \sum f(b_i)$ and thus
\begin{equation*}
\sum_{i=1}^k f(a_i)
\ge \sum_{i=1}^{\lfloor e^{\pi/M} -1 \rfloor} f(b_i) \\ = \lfloor
e^{\pi/M} -1 \rfloor \cdot f\left(\frac 1{e^{\pi/M} -1}\right) >
\left(e^{\pi/M} - 2\right)  M.
\end{equation*}
Since $M < \frac \pi{\log 4}$, we have $e^{\pi/M}>4$ and thus
\begin{equation*}
\sum_{i=1}^k \mu_i = \sum_{i=1}^k f(a_i) > M (e^{\pi/M} -2)
>     \frac 12 M e^{\pi/M}
\end{equation*}
as claimed.
\end{proof}


Let $\psi\colon(\Cbar\sm\diskbar)\to\Cbar$ be a linearizing map near
$\infty$ of $N_p$, i.e., $\psi(N_p(z))=\psi(z)(d-1)/d$ with
$\psi(\infty)=\infty$, and normalize so that $\psi(z)/z\to 1$ as
$z\to\infty$. Let
\[
W_R:=\{w\in\C\colon R(d-1)/d<|w|<R\}
\]
be a fundamental domain in linearizing coordinates.
\begin{lemma}
\label{Lem:AreaInRoundDomain}
For any channel $B_i$, we have
\[
|\psi(B_i)\cap W_R|\ge \Mod(B_i) R^2/ d^2
\;.
\]
\end{lemma}
\begin{proof}
This is another elementary exercise using extremal length:
fix a channel $B_i$ and let $B:=\psi(B_i)\cap W_R$. By conformal invariance, the modulus of $B_i$ equals the
modulus of $B$ where the boundaries are identified by multiplication
by $(d-1)/d$, and this is
\[
(\Mod B_i)^{-1} =(\Mod B)^{-1}= \sup_\rho\inf_\gamma
\frac{\ell^2(\gamma)}{\|\rho^2\|_{B}} \;,
\]
where $\rho\colon B\to\R^+$ are measurable functions,
$\gamma\colon[0,1]\to \ovl B$ are smooth curves with
$\gamma(1)=\gamma(0)(d-1)/d$, and
$\ell(\gamma)=\int_0^1\rho(\gamma(t))\,|\gamma'(t)|\,dt$.

We simply set $\rho\equiv 1|_B$ (the characteristic function of $B$).
If $A$ denotes the Euclidean area of
$B$, then $\|\rho^2\|_B=A$. The two boundary circles of $W_R$ have
radii $R$ and $R{(d-1)/d}$, so $\ell(\gamma)\ge R/d$.
Therefore, $1/\Mod B\ge R^2/d^2A$ or $A\ge \Mod (B) R^2/d^2=\Mod(B_i)R^2/d^2$.
\end{proof}

\begin{lemma}
\label{Lem:AreaInFundamentalDomain}
For $R\ge 5$, the intersection of the  annulus
\[
A_R=\left\{z\in\C\colon \frac{d-1}{d} R - \frac 1d<|z|<R\right\}
\]
with a channel of modulus $\mu$ has area at least
\[
\frac\mu {d^2}\cdot \frac{(R-1)^2(R-3)^2}{4(R+1)^2}
\;.
\]
\end{lemma}
\begin{proof}
Consider the circle $C_R:=\{z\in\C\colon |z|=R\}$, and the image
$C'_R:=N_p(C_R)$. Then $C'_R$ is another topological circle with absolute
values between $R(d-1)/d-1/d=R-(R+1)/d\ge (R-1)/2\ge 2$ and
$R(d-1)/d+1/d=R-(R-1)/d < R$.
Let $Z_R$ be the annulus
bounded by $C_R$ and $C'_R$; it is a fundamental domain for the Newton
dynamics, and we have $Z_R\subset A_R$. Consider a channel $B$ and
set $B_R:=B\cap Z_R$; this is a fundamental domain of the channel,
but not necessarily connected.

Consider again the linearizing function $\psi\colon
\Cbar\sm\diskbar\to \Cbar$ of $N_p$, normalized as
$\psi(\infty)=\infty$ and $\psi(z)/z\to 1$ as $z\to\infty$.
The Koebe distortion theorem in this normalization yields
\[
\frac{|z|-1}{|z|(|z|+1)} \le
\left|\frac{\psi'(z)}{\psi(z)}\right|
\le \frac{|z|+1}{|z|(|z|-1)}
\label{Eq:Koebe2}
\;.
\]

Define the sets
\[
B_{n}:=\left\{z\in B_R\colon R\left(\frac{d}{d-1}\right)^{n-1}
<|\psi(z)| < R\left(\frac{d}{d-1}\right)^{n}
\right\}
\]
for $n\in\Z$.
Each area element in $B_{n}$ is mapped into $W_R$ by the map
$z\mapsto \psi(z)((d-1)/d)^n$ with derivative
\[
|\psi'(z)| \left(\frac{d-1}{d}\right)^n
     < R \frac{|\psi'(z)|}{|\psi(z)|}  < \frac{R}{|z|} \cdot
\frac{|z|+1}{|z|-1} < \frac{2R}{R-1} \cdot \frac{R+1}{(R-3)}
\;,
\]
where we used the Koebe theorem in the second inequality and then
$|z|\ge (R-1)/2$.
This yields a diffeomorphism from $B_R$ to $\psi(B)\cap W_R$, except for discontinuities at the finitely many boundary arcs of the $B_n$.

The set $\psi(B)$ intersects $W_R$ in a set of area
$R^2\Mod(B)/d^2$ by Lemma~\ref{Lem:AreaInRoundDomain}, and areas in $B_{n}$ are
distorted by a factor of no more than the square of the derivative.
This implies that
\[
|B_R| > \frac{(R-1)^2 (R-3)^2}{4 d^2 (R+1)^2}
\Mod(B)
  \]
as claimed.
\end{proof}


\begin{lemma}
\label{Lem:ThinRootsSummary}
Let $R \ge 5$ and consider the annulus $A_R$ defined as in Lemma~\ref{Lem:AreaInFundamentalDomain}. 
Choose a probability $\rho\in(0,1)$. If
\[ \left\lceil 16 \pi d \frac{|\log(1-\rho)|+\log d}{M e^{\pi/M}} \cdot \frac{R (R+1)^3}{(R-1)^2(R-3)^2} \right\rceil\] points are randomly and independently
distributed in $A_R$, then for any polynomial $p\in\Pd$, each
thin root has at least one of these points in its immediate
basin with probability at least $\rho$.
\end{lemma}

\begin{proof}
The area of all channels within $A_R$ of any fixed thin root is at
least $\left((R-1)^2(R-3)^2/4d^2(R+1)^2\right)\sum\mu_i$ by
Lemma~\ref{Lem:AreaInFundamentalDomain}, and
$\sum\mu_i > \frac 12 Me^{\pi/M}$ by Lemma~\ref{Lem:area}. A simple
calculation shows that the area of  $A_R$ is less than $2\pi
R(R+1)/d$. Therefore, the probability that a point chosen randomly in
$A_R$ will lie in one of the channels of this root is at least
\[
q= \frac{Me^{\pi/M}}{16\pi d}\cdot \frac{(R-1)^2(R-3)^2}{R(R+1)^3}
\;.
\]
Now, suppose that we distribute some (large) number $K$ of points on
the annulus $A_R$, randomly and independently. Then the probability
that we do \textit{not} hit one of the channels of some fixed thin
root will be at most $(1-q)^K$. Since there are at most $d$ thin
roots, the probability that there is some thin root the channels of
which are not hit is at most $d (1-q)^K$. We need to make $K$ large
enough so that $d\,(1-q)^K < 1-\rho$, hence we need
$K>\log\big((1-\rho)/d\big)/\log(1-q)$.

Since $\log(1-q)<-q<0$, we have
\begin{align*}
{ \frac{\log\big((1-\rho)/d\big)}{\log(1-q) }} < \frac{ \log(1-\rho)-\log d}{-q } = \frac{|\log(1-\rho)|+\log d}{q} \\
= 16\pi d { |\log(1-\rho)|+\log d\over Me^{\pi/M}} \cdot
\frac{R(R+1)^3}{(R-1)^2(R-3)^2}
\;,
\end{align*}
so it suffices to distribute this number of points
within the annulus at random so that, with probability at least
$\rho$, at least one channel of each thin root is hit.
\hide{
Since $\log(1-q)<-q$ and $\log(1/2d) > \log(1/d^2) = -2 \log d$, we have
\[  \frac{\log\left(\frac 1 {2d} \right)}{\log(1-q)} < \frac{
\log{\frac 1 {d^2}} }{-q} = \frac{2 \log d}{q} = \frac{32 \pi d \log
d}{M e^{\pi/M}} \cdot \frac{R (R+1)^3}{(R-1)^2(R-3)^2} \;,  \]
so it suffices to distribute this number of points
within the annulus at random so that, with probability at least
$1/2$,  at least one channel of each thin root is hit.
}
\end{proof}

\begin{Remark}
\label{Rem:R}
Increasing the radius $R$ will decrease the necessary number of
points to asymptotically $16\pi d\big(|\log(1-\rho)|+\log d\big)/Me^{\pi/M}$ for large $R$.
The disadvantage is that the required number of iterations will be
very large until the roots are reached.
In this article, we do not optimize the number of starting points vs.\ the number of iterations: indeed, it is possible to optimize all constants by refining several of our estimates (see below).
\end{Remark}

\section{Conclusion}

\begin{proof}[Proof of Theorem~\ref{Thm:Main}]
We have to distribute $16\pi d/{M^2}$ points within the annulus $V$ by
the algorithm described in Section \ref{Sec:thick} to be sure that all thick roots are
found. To hit the thin roots, we consider the annulus $A_R$ defined as in Lemma \ref{Lem:AreaInFundamentalDomain}, where we choose $R=11$ (see Remark \ref{Rem:R}) so that $R(R+1)^3/(R-1)^2(R-3)^2= 2.97$; in order to hit find all the thin roots with probability at least $\rho=1-1/d$, we thus have to randomly distribute 
\begin{align*}
16\cdot 2.97\pi d \big(|\log(1-\rho)|+\log d\big)/M
e^{\pi/M}<300d\log d/Me^{\pi/M}
\end{align*}
points inside the annulus $A_R$ (in both statements, we ignored the
condition that we need to round up certain numbers).

This gives us a total of
\[
P(M) =  \frac {16\pi d} {M^2} + \frac{300 d \log d}{M e^{\pi/M}}
\]
points to be chosen
to hit the channels of all roots with
probability at least $1-1/d$.
In particular, setting $M = \pi/\log\log d$, it suffices to use at least
\[
P\left(\frac \pi{\log\log d}\right)
=
\frac{16 d}\pi (\log\log d)^2 + \frac{300}{\pi} d \log\log d
=
O(d\,(\log\log d)^2)
\]
points.
\end{proof}

\begin{Remark}
\label{Rem:smallValues}
Strictly speaking, this proof only works for $d>e^4 \approx 54.6$ as we
claimed in the beginning that $M < \pi/\log 4$ and finally chose
$M=\pi/\log\log d.$ However, we only need this to simplify some term
in the proof of Lemma \ref{Lem:area}; for $2\le \log d < 4$, by
being a little bit more careful in the proof of Lemma \ref{Lem:area}
one can even get slightly better constants, whereas for $1\le \log d
\le 2$ one has to choose another value for $M$ to get the same final
upper bound.
\end{Remark}

\begin{Remark}
\label{higherProb}
Of course, the probability $1-1/d$ can be replaced by any
probability $\rho\in(0,1)$ by appropriately increasing the number of 
points. For $M=\pi/\log\log d$, the number of points to find the thin 
roots then becomes $O\big(d\log\log d(1+|\log(1-\rho)|/\log d \big)$.
Including thick roots as well, and ignoring dominated terms, the 
total number of points becomes
\begin{align*}
O\left(d(\log\log d)^2+d\log\log d|\log(1-\rho)|/\log d\right)
\\
< O\left(d(\log\log d)^2+d|\log(1-\rho)|\right)
\;.
\end{align*}
This will not even change the leading term of the number of points as 
long as $\rho\le 1-1/d^{\log d \log\log d}$. 
\hide{
If $\rho$ and $d$ are 
considered as independent quantities, the asymptotics simplifies to 
$O\left(d(\log\log d)^2+d|\log(1-\rho)|\right)$.
}
\hide{
if we carry out the algorithm for thin roots described above several times
independently. This will not even change the leading order of the
complexity of the whole algorithm as the thin roots only account for
a relatively small part of the points needed (a factor $\log\log d$
less than the thick roots). In fact, adapting the proof of Lemma~\ref{Lem:ThinRootsSummary} in the obvious way, one can do a little better:
the numerator $32\pi d\log d$ for $\rho=1/2$ then becomes $16\pi d(\log d+|\log(1-\rho)|)$. For $M=\pi/\log\log d$, the number of points to find the thin roots then becomes $O\big(d\log\log d+d|\log(1-\rho)| \big)$.
}
\end{Remark}

\begin{Remark}
At several places, we preferred the simple argument over optimal
numerical values, as far as constant factors were concerned. If one
were to optimize these factors, it would involve the following places.
The thick roots have the higher complexity, so asymptotically it is
most important to optimize constants here.
In Lemma~\ref{Lem:UniversalGeometry}, the modulus of a quadrilateral is
estimated only roughly using a simple argument. The precise value of
this quadrilateral can be determined using elliptic integrals; this
has been done in \cite{HSS} in an analogous situation. One could then
optimize the number of circles and the number of points on them:
taking more (or fewer) circles would allow us to use fewer (more)
points on each of them, and there is an optimal value of circles that
minimizes the total number of points.

For thin roots, we used the estimate $e^{\pi/M}-2>e^{\pi/M}/2$ at the
end of the proof of Lemma~\ref{Lem:area}, and for large $d$ this
loses a factor of $2$. Moreover, in the proof of
Lemma~\ref{Lem:ThinRootsSummary} one could gain a factor of $2$ by using a fixed probability $\rho$, rather than $\rho=1-1/d$.
Finally, there is a certain
loss in the estimation of probabilities of hitting the $d$ different
basins; these probabilities are not quite additive as estimated. Our
estimates in the thin case are roughly a factor $4$ away from
being optimal. And of course, one can reduce the radius $R$ of the starting points, and thus the required number of iterations, at the expense of increasing the number of starting points.
\end{Remark}

\begin{Remark}
Since the complexities of the deterministic and the probabilistic
parts are different, it is tempting to reduce the total complexity by
choosing a  value of $M$ different from $\pi/\log\log d$ so that both
partial complexities become closer to each other. Slight improvements
are indeed possible that way, but the gain seems to be minimal.
{For example, one has
\[
P\left(\frac{\pi}{(\log\log d)^{1-1/(1+\log\log d)}}\right) =
O\left(d\,(\log\log d)^{2-2/(1+\log\log d)}\right).
\]
In this case, the deterministic term is still much bigger than the
probabilistic one. Such calculations seem to become much more
complicated with relatively little gain.
}

Moreover, we have not used the condition  $\sum_{\alpha_i}k_i\le
2d-2$ coming from the total number of ``free'' critical points. We
believe that the effect of incorporating this condition will be
marginal.
\end{Remark}


\begin{thebibliography}{HSS}
\bibitem[A]{Ahlfors} Lars Ahlfors, \emph{Lectures on quasiconformal
mappings}, Second edition. University Lecture Series  \textbf{38}.
American Mathematical Society, Providence, RI, 2006.

\bibitem[ABS]{NewtonTodor}
Magnus Aspenberg, Todor Bilarev, Dierk Schleicher:
\emph{On the speed of convergence of Newton's method for complex polynomials}. Manuscript, in preparation.

\bibitem[BC]{BuffCheritat}
Xavier Buff, Arnaud Ch\'eritat, \emph{Ensembles de Julia quadratiques de
mesure de Lebesgue strictement positive}. C. R. Acad. Sci. Paris
\textbf{341} 11 (2005), 669--674.

\bibitem[DH]{PolyLike}
Adrien Douady, John Hubbard, \emph{On the dynamics of polynomial-like
maps}. Ann. Sci. Ec. Norm. (4) \textbf{18} 2 (1985), 277--343.

\bibitem[HLP]{Ineq}
Godfrey H. Hardy, John E. Littlewood, George P\'olya, \emph{Inequalities},
Second edition. Cambridge University Press, Cambridge 1952.

\bibitem[HSS]{HSS}
John Hubbard, Dierk Schleicher, Scott Sutherland, {\em How to find
all roots of complex polynomials by Newton's method}. Inventiones
Mathematicae {\bf 146} (2001), 1--33.

\bibitem[Mi]{ZhenyaThesis}
Yauhen Mikulich, \emph{Newton's Method as a Dynamical System}. 
Thesis, Jacobs University Bremen, 2011.

\bibitem[M]{MiBook}
John Milnor, \emph{Dynamics in one complex variable}, third edition.
Annals of Mathematics Studies \textbf{160}, Princeton University
Press, Princeton, NJ, 2006.

\bibitem[Pr]{Prz} Feliks Przytycki, \emph{Remarks on the simple
connectedness of basins of sinks for iterations of rational maps}.
In: Dynamical Systems and Ergodic Theory, ed. by
K. Krzyzewski, Polish Scientific Publishers, Warszawa (1989), 229--235.


\bibitem[R\"u]{Johannes} Johannes R\"uckert,
\emph{Rational and transcendental Newton maps}. In: Holomorphic
dynamics and renormalization,  Fields Inst. Commun. \textbf{53},
Amer. Math. Soc., Providence, RI, 2008, pp. 197--211.

\bibitem[Sch1]{NewtonEfficient} Dierk Schleicher, \emph{Newton's
method as a dynamical system: efficient root finding of polynomials
and the Riemann $\zeta$ function}. In:  Holomorphic dynamics and
renormalization, M. Lyubich and M. Yampolski (eds,) Fields Inst.
Commun. \textbf{53}, Amer. Math. Soc., Providence, RI, 2008,
pp.~213--224.

\bibitem[Sch2]{NewtonGlobal} Dierk Schleicher, \emph{On the efficient global dynamics of Newton's method for complex polynomials}.
Manuscript, submitted.

\bibitem[S1]{Smale} Stephen Smale, \emph{On the efficiency of
algorithms of analysis}. Bulletin of the American Mathematical
Society (New Series) \textbf{13} 2 (1985), 87--121.

\bibitem[S2]{Smale2} Stephen Smale, \emph{Newton's method estimates 
from data at one point}, in:  The Merging Disciplines: New Directions 
in Pure, Applied and Computational Mathematics, Springer-Verlag, 
Berlin, New York (1986), 185--196.


\end{thebibliography}
\end{document}